\documentclass[oneside, 10pt]{amsart}

\usepackage{amsmath, amsfonts, amssymb, amsthm, graphics}

\newtheorem{theorem}{Theorem}[section]

\newtheorem{proposition}[theorem]{Proposition}

\theoremstyle{definition}

\newtheorem{remark}[theorem]{Remark}

\numberwithin{equation}{section}

\begin{document}

\title{On holomorphic functions on negatively curved  manifolds}

\author{Marijan Markovi\'{c}}

\begin{abstract}
Based on a well known     Sh.-T. Yau  theorem we obtain that the real part of a holomorphic function on a K\"{a}hler manifold with the Ricci curvature
bounded from below by $-1$ is contractive with respect to the distance on the manifold and the hyperbolic distance on $(-1,1)$ inhered from the domain
$(-1,1)\times\mathbb{R}$. Moreover, in the case of bounded holomorphic functions we prove that the modulus is contractive with respect to the distance
on the manifold and the hyperbolic distance  on the unit disk.
\end{abstract}

\address{Faculty of Sciences and Mathematics\endgraf
University of Montenegro\endgraf
D\v{z}ord\v{z}a Va\v{s}ingtona bb\endgraf
81000 Podgorica\endgraf
Montenegro\endgraf}

\email{marijanmmarkovic@gmail.com}



\maketitle

\section{Introduction}

\subsection{The Yau generalization of the Schwarz--Pick lemma}

The result concerning holomorphic mappings on K\"{a}her manifolds given in the following proposition is well known and  proved by Yau \cite{YAU.AJM},
in 1978. We refer also to the work of Royden \cite{ROYDEN} for certain improvements of this result and the Kobayashi book \cite{KOBAYASHI.BOOK} as a
general reference for holomorphic mappings between complex manifolds where the questions on distance, area  and volume decreasing  properties    are
considered.

\begin{proposition}[Cf. \cite{YAU.AJM}]\label{PROP.YAU}
Let     $M$ be a complete K\"{a}hler manifold with the Ricci curvature bounded from below by a negative constant $K_M$. Let $N$ be another Hermitian
manifold with  holomorphic sectional curvature bounded from above by a negative constant $K_N$. Then any holomorphic mapping from $M$ into  $N$ does
not increase distances  more  than a factor depending only on the curvatures of $M$ and $N$. The  factor is  $ {K_M}/{K_N}$.
\end{proposition}

In the rest of the paper we denote by $M$ the kind of the Hermitian manifold as in the above proposition. On the other hand, the manifold $N$ may be,
for example, a bounded homogenous domain in $\mathbb {C}^n$       equipped with the Bergman metric, since a such one domain has constant holomorphic
sectional curvature equal to $-1$  \cite{BERGMAN}; in particular it may be the unit ball in $\mathbb{C}^n$.

From Proposition \ref{PROP.YAU} it follows that if the Ricci curvature of $M$ is bounded from below by $-1$, and if the holomorphic section curvature
of $N$ is bounded from  above by $-1$,  then any  holomorphic mapping $f:M\to N$ is distance decreasing, i.e.,
\begin{equation*}
d_N (f(z),f(w))\le d_M (z,w),\quad z,w\in M,
\end{equation*}
where $d_M$ is the distance on $M$, and $d_N$ is the distance on $N$.

Recall that the distance on a Hermitan manifold $N$ is given by
\begin{equation*}
d_N (z,w)= \inf_\gamma \ell (\gamma),
\end{equation*}
where $\gamma:[a,b]\to N$ is among partially $\mathcal {C}^1$-smooth paths on $N$ with endpoints at $z$ and $w$, i.e., $\gamma(a)=z$ and $\gamma(b)=w$.
We have denoted by $\ell (\gamma)$  the length  of the path $\gamma$ with respect to the Hermitian form $H_z = H^N_z$, $z\in N$ on $N$  which is
\begin{equation*}
\ell (\gamma) =   \int_a^b \sqrt{ H_ { \gamma(t) }  ( \gamma'(t),\gamma'(t)  ) } dt.
\end{equation*}
Let us note that if  $\gamma:[a,b]\to M$ is a partially  $\mathcal{C}^1$-smooth path parameterized by arc length, then  since
\begin{equation*}
 \ell (\gamma|_{[a,t]})=\int_a^t  \sqrt{ H_ { \gamma(t) }  (\gamma'(t),\gamma'(t) )  }dt =t-a,\quad t\in[a,b],
\end{equation*}
we have  $H_{\gamma (t)} (\gamma' (t), \gamma' (t))=1$, $t  \in [a,b]$.

If $N\subseteq \mathbb{C}$ is a surface (that is $1$-dimensional Hermitian manifold), the holomorphic section curvature are equal to the Gauss curvature
of $N$, which is
\begin{equation*}
- \frac{\Delta  \log h(z)}{h(z)^2},
\end{equation*}
where the Hermitian form $H^N_z$ on  $N$ is represented as $H^N_z(u,v) = h(z)^2 u\overline{v}$, $z\in N$, $u,v\in\mathbb{C}$, $h(z)>0$.

The result stated in the above proposition   is a  generalization of the classical Schwarz--Pick lemma which says that a holomorphic function which maps
the unit disk $\mathbb {D} \subseteq\mathbb{C}$ into itself does  not  increase  the hyperbolic distance $\sigma $ on $\mathbb {D}$, i.e.,
\begin{equation*}
\sigma(f(z), f(w))\le\sigma(z,w),\quad  z,w\in \mathbb{D};
\end{equation*}
recall that the hyperbolic metric on  $\mathbb {D}$  is  given by the following  Hermitian form
\begin{equation*}
H^{\mathbb{D}}_z  (u,v)= \frac{4u\overline{v}} {(1-|z|^2)^2} ,\quad  z\in \mathbb {D},\, u, v\in \mathbb {C},
\end{equation*}
with  the Gauss curvature equal to $-1$.

Proposition \ref{PROP.YAU} is also generalization of  the Ahlfors result for surfaces with nonpositive Gauss curvature  \cite{AHLFORS.TAMS, AHLFORS.BOOK}.

\subsection{The main  results of this work}
Thought the rest of this paper let $J = (a,b)\subseteq\mathbb {R}$ be an interval, bounded or unbounded but not equal to $\mathbb {R}$. The weight on $J$
is any positive and  continuous function on this interval.      If $\omega$ is a weight on $J$, we define the $\omega$-distance between $a,b\in J$ in the
following way
\begin{equation*}
d_\omega(a,b) = \int_{a}^{b}\omega (t)dt.
\end{equation*}
For a   weight $\omega\in  \mathcal {C}^2 (J)$ we introduce  the following quantity
\begin{equation*}
k_\omega (t) = \frac{\omega'(t)^2-\omega(t)\omega''(t)}{\omega(t)^4},\quad t\in J.
\end{equation*}

The  first main result of this paper  is the distance decreasing property of the real part of a holomorphic function on the manifold $M$ with respect to
the distance on $M$ and the $\omega$-distance on $J = (a,b)$ provided that $k_\omega\le -1$. Therefore, if  $J$ is equipped  with a weight $\omega$ with
$k_\omega\le -1$ (of special interest for us  are weights    with  $k_\omega \equiv -1$  on $J$),  the  result says that  the real part of a holomorphic function  $f:M\to J$ is distance  decreasing, i.e., we have
\begin{equation*}
d_\omega (\mathrm {Re}f(z),\mathrm {Re}f(w))\le d_M (z,w),\quad z,w\in M.
\end{equation*}

This improves some recent results \cite{CHEN, KV.PAMS, MATELJEVIC.JMAA}  obtained for real--valued harmonic functions on proper simply connected domains
in  $\mathbb{C}$, since on a  such type domain a harmonic function is a real part of a holomorphic one. See also the  recent papers
\cite{MARKOVIC.ARCHM, MATELJEVIC.SANU, MATELJEVIC.SVETLIK.AADM}.

In the second part of the next section we  prove that the modulus of a bounded holomorphic function is distance decreasing with respect to the distance
on $M$  and  the hyperbolic distance $\sigma $  on  $\mathbb{D}$. More precisely,  this result says that
\begin{equation*}
\sigma ( |f|(z),|f|(w))\le d_M (z,w),\quad z,w\in M,
\end{equation*}
provided that  $f$ is bounded by $1$.

\section{Distance decreasing property  of the real part and the modulus}

\subsection{Distance decreasing  property  of the real part of a holomorphic function}
Our  first main result follow from the following theorem.

\begin{theorem}\label{TH.MAIN}
Let $f$ be a holomorphic function on $M$ such that $\mathrm {Re} f\in J$. Let $\omega$ be a weight on $J$ which satisfies $k_\omega \le -1$.      Then
$\mathrm {Re} f$ is  distance decreasing, i.e., we have
\begin{equation}\label{EQ.RE.DISTANCE.DECREASING}
d_\omega  (\mathrm {Re} f  (z), \mathrm {Re} f (w)) \le    d_M (z,w),\quad z,w\in M.
\end{equation}
\end{theorem}

\begin{proof}
Note that $f$ maps $M$ into the vertical strip domain $N = J\times\mathbb{R}\subseteq \mathbb{C}$. Introduce the Hermitian form on $N$ in the following
way
\begin{equation*}
H^N_z(u,v) = \tilde {\omega}(z)^2u\overline{v},\quad u,v\in \mathbb{C},\, \tilde {\omega}(z) = \omega (\mathrm {Re}\, z),\, z\in N.
\end{equation*}
Let $d_N$ be the distance on $N$.  According to our assumption, the Gauss curvature of $N$ is bounded  by $-1$ form above. Indeed, for $z\in N$ we have
\begin{equation*}
-\frac { \Delta \log \tilde {\omega}   (z)}{\tilde {\omega}^2(z)}=
\frac { {\omega} '(\mathrm {Re}\, z)^2 -   {\omega}(\mathrm {Re}\, z)   {\omega}''(\mathrm {Re}\, z)}
{{\omega}(\mathrm {Re}\, z)^4} = k_\omega (\mathrm {Re}\, z)  \le -1.
\end{equation*}

In view of Proposition \ref{PROP.YAU} the function  $f:M\to N$ is distance decreasing, which implies that
\begin{equation}\label{EQ.H}
\sup_{\zeta\in\mathbb{C}^m,\,  H_z ( \zeta, \zeta )=1}
\left| \left<\nabla f(z),{\zeta} \right> \right|  \le \frac {1  }  {\tilde {\omega}(f(z))},\quad z\in M
\end{equation}
in local coordinates $z_1,z_2,\dots, z_m$ around $z$, where $\nabla f=(f_{z_1},f_{z_2},\dots,f_{z_m})$ and $m=\dim M$. Indeed, let $\gamma:[a,b]\to M$
be a smooth path parameterized by arc length, and let $\gamma'(a ) = \overline {\zeta}$. Since $f:M\to N$  is distance decreasing, we have
\begin{equation*}
d_N ( f( \gamma (t) ), f ( \gamma (a) ) )\le d_M (\gamma (t), \gamma (a)) \le   \ell (\gamma|_{[a,t]}),\quad t\in [a,b].
\end{equation*}
If $\nabla f (z)\ne0$, then for $t$ sufficiently close to $a$ there holds  $f( \gamma (t))\ne f(\gamma (a) )$.  It follows that
\begin{equation*}
\frac {d_N  ( f(\gamma (t) ),  f( \gamma (a) ) ) }  {| f(\gamma (t) )- f( \gamma (a) ) |} \frac {| f(\gamma (t) ) - f( \gamma (a) ) |}  {t-a}
\frac{t-a}{ \ell (\gamma|_{[a,t]})}\le  1.
\end{equation*}
Having  in mind  that  $\lim_{\eta\to \zeta}\frac{d_N (\zeta,\eta)}{|\zeta-\eta|}=\tilde {\omega}(\zeta)$ (for this equality in a general setting, i.e.,
for vector spaces with norm,  we refer to  \cite{MARKOVIC.JGA}),   from the last inequality we obtain \eqref{EQ.H}.

Let now $z,w\in M$ be arbitrary,  and let $\gamma:[a,b]\to M$ be a path. Assume that $\gamma$ is  $\mathcal {C}^1$-smooth parameterized by arc length, so
that we have $\ell(\gamma) = b-a$, and  contained in a single chart of $M$ (in general case we should consider pieces of $\gamma$ and apply the following
rezoning on the each part of the path). Denote $\mathbf{t}_ \gamma(t) = \overline{\gamma'(t)}$. Then we have
$H_{\gamma(t)}(\mathbf{t}_ \gamma(t),\mathbf{t}_ \gamma(t)) =H_{\gamma(t)}(\gamma'(t),\gamma'(t))=1$.  Note that $\delta =  \mathrm {Re} f \circ  \gamma$
may be considered as a $\mathcal {C}^1$-smooth path in $J$ with endpoints  $\mathrm {Re} f (z)\in J$ and $ \mathrm {Re} f(w)\in  J$.     Therefore, using
\eqref{EQ.H} in the second inequality below, for the $\omega$-distance between $\mathrm{Re} f(z)$  and  $\mathrm{Re} f (w)$ we obtain
\begin{equation*}\begin{split}
d _\omega(  \mathrm{Re} f(z)  ,\mathrm{Re} f (w) )&   \le \int_a^b \omega (\delta (t))  |\delta'(t) |  dt
= \int_a^b \omega ( \mathrm{Re} f (\gamma(t) ) \left|\frac d{dt}  \mathrm{Re } f   (\gamma(t) ) \right |dt
\\&  = \int_a^b   \omega ( \mathrm{Re} f (\gamma(t) ) \left| \mathrm{Re}\frac d{dt} f (\gamma(t) ) \right|dt
\\&= \int_a^b \omega  ( \mathrm{Re}f (\gamma(t) )  \left|\mathrm{Re} \left< \nabla f ( \gamma (t) ),\mathbf{t}_{\gamma}(t)\right>\right| dt
\\&\le  \int_a^b  \omega  ( \mathrm{Re} f (\gamma(t) )  \left| \left< \nabla f ( \gamma (t) ), \mathbf{t}_{\gamma}(t)\right>\right|   dt
\\&\le  \int_a^b \frac { \omega (  \mathrm{Re}(     f ( \gamma(t) ) )} {\tilde {\omega} ( f (\gamma(t)))}dt = b-a =  \ell (\gamma),
\end{split}\end{equation*}
since  $\tilde {\omega} (z) = {\omega}(\mathrm {Re}\, z) $. If we take infimum over all $\gamma$, we obtain the distance decreasing  property of the real
part of $f$, i.e., \eqref{EQ.RE.DISTANCE.DECREASING}.
\end{proof}

\begin{remark}
Having  in mind the preceding theorem we should  find   positive   solutions of the  differential equation
\begin{equation*}
\frac {\omega '^2 - \omega\omega '' }{\omega  ^4}  = -k^2,\quad k\ge 1.
\end{equation*}

This differential equation may be rewritten as $\frac {(\log \omega)''}{\omega^2} = k^2$. If we introduce $\lambda = \log(2k^2\omega^2)$, then we  obtain
the following equation  $\lambda'' = e^{\lambda}$.  The last equation may be solved in the standard way.  One obtains the general solutions:
\begin{equation*}
e^{\lambda(t)}{\sin^2 (C_1 t+C_2)} =  {2C_1^2},\quad  e^{\lambda(t)}{\mathrm {sh}^2 (C_1 t+C_2)} =  {2C_1^2},\quad e^{\lambda(t)}(t+C)^2 =   2,
\end{equation*}
where $C_1$, $C_2$ and $C$ are  constants.

Therefore, the positive solutions of the equation  $\frac {\omega '^2 - \omega\omega '' }{\omega  ^4}  = -k^2$ on $J$ are:
\begin{itemize}
\item $\omega(t) = \frac{C_1}{k |\sin (C_1 t+C_2)|}$;
\item $\omega(t) = \frac{C_1}{k|\mathrm {sh} (C_1 t+C_2)|}$;
\item $\omega(t) = \frac 1 {k|t+C|}$,
\end{itemize}
\noindent where the constants $C_1>0$, $C_2$ and $C$ should  be  adjusted in a such way that the interval $J$ is contained in the  domain of $\omega$.
\end{remark}

An immediate consequence  of the preceding  theorem  and the   remark  is  the following theorem.

\begin{theorem}
(i) Let $f$ be a holomorphic function on the manifold  $M$ with the  real part   bounded by $1$. Then $\mathrm {Re} f$ is contractive with respect to
the distance  on  $M$ and the hyperbolic distance on $(-1,1)$ inhered from the domain $(-1,1)\times \mathbb{R}$.

(ii) Let $f$ be a holomorphic  function   on the manifold  $M$  such that    the real part of $f$ is positive. Then  $\mathrm {Re} f$ is contractive
with  respect to the  distance on $M$ and the hyperbolic distance   on $(0,\infty)$ inhered from $(0,\infty)\times \mathbb {R}$.
\end{theorem}

\begin{proof}
(i) The  hyperbolic metric on  $(-1,1)\times \mathbb {R}$ is given by the Hermitian form
\begin{equation*}
H_z (u,v) =\left(\frac \pi 2 \frac {1} {\cos ( \frac \pi 2 \mathrm{Re}) \, z}\right)^2 u\overline{v},
\quad z\in (-1,1)\times \mathbb {R},\,  u,v\in\mathbb{C}.
\end{equation*}
The corresponding weight on $(-1,1)$ is $\omega (t) = \frac {\pi}2 \frac{1}{\cos (\frac {\pi t} 2 )}$ which is the first solution given in the preceding
remark  for $k=1$,  $C_1=\frac\pi2$, and  $C_2=-\frac\pi2$.

(ii) The hyperbolic metric on the  half--plane $\{z\in\mathbb {C}:\mathrm {Re}\, z>0\}$  is
\begin{equation*}
H_z (u,v) = \frac {1} {( \mathrm {Re}\, z)^2 }  u\overline{v},\quad  u,v\in\mathbb{C},
\end{equation*}
so the  weight on $J= (0,\infty)$ is  given by $\omega (t)  =  \frac 1t$, which  is the third solution for $k=1$ and  $C =0$.
\end{proof}

\begin{remark}
We would like here to compare the result  given in  above theorem  with some recent results for real--valued harmonic functions.     Let us replace the
manifold $M$ with the  unit disk $\mathbb {D}$ equipped with the hyperbolic metric (with the Gauss curvature equal to $-1$). Then the first part of the
above  theorem  recovers the result obtained  by Chen \cite{CHEN, PAVLOVIC.BOOK}.

Note that the weight $\omega(t)=\frac {\pi}2 \frac{1}{\cos (\frac {\pi t} 2 )}$ on $(-1,1)$ is comparable  to the  following one $\tilde {\omega} (t) =
\frac 2{1-t^2}$ on the same interval in the sense that $\frac \pi4 \tilde{\omega} \le \omega$ on $(-1,1)$. Indeed, this inequality is elementary and it
may be found in  \cite{CHEN} or \cite{MATELJEVIC.JMAA}. This means that the distance $d_\omega (t,t')$ is greater then
$\frac \pi 4 d_{\tilde{\omega}} (t,t')$ for $t, t'\in (-1,1)$. Thus  we have  just recovered the result of Kalaj and Vuorinen \cite{KV.PAMS} which says
that for a harmonic function   $U:\mathbb{D}\to (-1,1)$ there holds
\begin{equation*}
\sigma (U (z), U(w))\le \frac 4\pi \sigma(z,w),\quad   z,w \in \mathbb {D},
\end{equation*}
since $\tilde{\omega}$-distance on $(-1,1)$ is equal to the distance on  $(-1,1)$ inhered from the hyperbolic distance $\sigma$ on $\mathbb {D}$.  This
improves also \cite{KALAJ.GLASGOW}.

On the other hand, the part (ii)  gives the  result of Markovi\'{c} \cite{MARKOVIC.INDAG} for positive harmonic functions on the unit disk. See also
\cite{MELENTIJEVIC.AASF}.
\end{remark}

\subsection{Distance decreasing  property   of the modulus of a holomorphic function}
Having in mind the preceding theorem  on the contraction  of the real part of a holomorphic function, one may pose the natural question what  can be said
about the modulus $|f|$ of a bounded holomorphic function $f$ on the manifold $M$. In the rest of this paper we will consider the question concerning the
the distance contraction of  the modulus of a bounded  holomorphic  function on the manifold $M$ with respect to the distance on $M$ and the hyperbolic
distance on  $\mathbb{D}$.

This question is also motivated by the Pavlovi\'{c} work \cite{PAVLOVIC.PAMS} on the Schwarz lemma for the modulus of a holomorphic mapping   on the disk
$\mathbb{D}$ with values in the unit ball $\mathbb {B}^n$.  Note that $|f|$ is not necessary (real) differentiable on $\mathbb {D}$, so the Schwarz lemma
for the modulus of a holomorphic function  has to be  formulated in the  following way
\begin{equation}\label{EQ.MODULUS}
\nabla^\ast  |f| (z) \le \frac{1-|f(z)|^2}{1-|z|^2},\quad z\in \mathbb {D},
\end{equation}
where  for a real--valued  function  $g$  on $\mathbb{D}$ we have denoted
\begin{equation*}
\nabla^\ast g(z)  =   \limsup _{h\to 0 }\frac {|g(z+h)|}{|h|}.
\end{equation*}
If $f(z)\ne 0$, then $|f|$ is differentiable  at $z\in \mathbb {D}$, and  we may replace  $\nabla^\ast |f| (z)$ with the modulus of the ordinary gradient
$|\nabla   |f| (z)|$.

The Pavlovi\'{c} result is stated and proved in the form we have just mentioned.       The lemma is used to extend some results of K. Dyakonov on modulus
of continuity of holomorphic functions on the disk. However, it may be rewritten in the  manner which does not involve  any kind of derivative  by saying
that  $|f|$  is contractive in the  hyperbolic distance $\sigma$ on $\mathbb {D}$, i.e.,
\begin{equation*}
\sigma  (|f| (z), |f| (w)) \le \sigma (z,w),\quad z,w\in \mathbb {D}.
\end{equation*}
Indeed, this follows since \eqref{EQ.MODULUS} may be rewritten as
\begin{equation*}
\frac {2 \nabla^\ast  |f| (z) }{1-|f(z)|^2}\le\frac{2}{1-|z|^2}
\end{equation*}
(see \cite{MARKOVIC.JGA} for a similar result in a general setting, i.e.,              for Fr\'{e}chet differentiable mappings between vector spaces with
norm).  In the sequel we will   consider a generalization of the Pavlovi\'{c} result  for   holomorphic functions  on the  manifold $M$.

Recall that the  pseudo--hyperbolic distance $\rho $ on  $\mathbb {D}$   is given by
\begin{equation*}
\rho (z,w) = |\varphi _z (w)|,\quad z,w\in \mathbb {D},
\end{equation*}
where  $\varphi_z (w)= \frac{z-w}{1-\overline {z}w}$, is a  conformal transformation of  $\mathbb {D}$ onto itself.

Duren and Weir \cite{DUREN.TAMS}       showed that the pseudo-hyperbolic distance (even on the unit ball $\mathbb{B}^n$ in $\mathbb{C}^n$) satisfies the
inequality
\begin{equation}\label{EQ.ABS.RHO}
\rho (|z|,|w|)\le  \rho (z,w),\quad z,w\in \mathbb {D}.
\end{equation}
They gave an interesting geometric  proof (in any dimension). For the sake of completeness we will prove this inequality directly having in mind a  well
known identity
\begin{equation*}
1-|\varphi_z (w)| = \frac {(1-|z|^2) (1-|w|^2)}{|1-\overline{z}w|},\quad z,w\in\mathbb{D}
\end{equation*}
(this proof may be adapted for the unit ball). Using the simple  inequality  $|1-\overline{z}w|\ge  |1- |\overline{z}w | | = 1-|z||w|$,  we obtain
\eqref{EQ.ABS.RHO}.

Since  $t\to \log \frac{1+t}{1-t }$   is an increasing function on $(-1,1)$, and since
\begin{equation*}
\sigma    (z,w ) =  \log  \frac{1+\rho (z,w)}{1-\rho (z,w)},
\end{equation*}
we may conclude that the hyperbolic   distance on the  unit disk   satisfies  the inequality
\begin{equation}\label{EQ.ABS.HYPERBOLIC}
\sigma (|z|,|w|)\le \sigma  (z,w), \quad z,w\in \mathbb{D}.
\end{equation}

\begin{theorem}
Let $f$ be a  holomorphic function on the manifold   $M$  such that $|f (z)|<1$, $z\in M$. Then   we have
\begin{equation*}
\sigma  (|f| (z),|f| (w))\le d_M (z,w),\quad    z, w \in M.
\end{equation*}
In other words,   the  modulus  of $f$ is  contractive with respect  to the distance  on   $M$ and the hyperbolic distance on  $\mathbb {D}$.
\end{theorem}

\begin{proof}
Using \eqref{EQ.ABS.HYPERBOLIC} we obtain
\begin{equation*}
\sigma   (|f|  (z), |f| (w))  \le  \sigma   ( f(z), f (w))\le   d_M(z,w),\quad z, w \in M,
\end{equation*}
which we aimed to prove.
\end{proof}

\begin{remark}
That the modulus of a holomorphic mapping on $M$ with values in the unit ball $\mathbb {B}^n\subseteq \mathbb{C}^n$ is contractive (the above  mentioned
Pavlovi\'{c} result) one may see in the same way as in the proof of the preceding  theorem.     Recall that the  Bergman metric on  $\mathbb {B}^n$ is a
Hermitian form  given by
\begin{equation*}
H_z(u,v) =2  \frac{(1-|z|^2)\left<u,\overline{v}\right>+\left<u,\overline{z}\right> \left<v,\overline{z}\right>}{(1-|z|^2)^2},
\quad z\in \mathbb {B}^n,\,   u,v\in \mathbb {C}^n.
\end{equation*}
The  holomorphic  sectional curvature of the Bergman  metric is constant and equal to $-1$.
Therefore, by  the Yau theorem any  holomorphic mapping $f:M\to\mathbb {B}^n$ is distance decreasing. On the other hand, the explicit form of the Bergman
distance on  $\mathbb{B}^n$ is
\begin{equation*}
\beta (z,w) =  \log  \frac{1+|\varphi _z (w)|}{1- |\varphi _z (w)|}= \log  \frac{1+\rho (z,w)}{1-\rho (z,w)},\quad z, w\in\mathbb{B}^n,
\end{equation*}
where
\begin{equation*}
\varphi_z (w)= \frac
{ z- \frac{\left<w,z\right>}{|z|^2} - \sqrt{1-|z|^2} \left( w - \frac{\left<w,z\right>}{|z|^2}\right) }
{1-\left<w,z\right>}
\end{equation*}
is a bi-holomorphic transformation of  $\mathbb{B}^n$ onto itself, and $\rho (z,w)=|\varphi _z (w)|$ is the pseudo--hyperbolic distance on $\mathbb {B}^n$.

For the unit disk $\mathbb {D}\subseteq \mathbb {B}^n$ the restricted  Bergman distance $\beta |_{\mathbb{D}}$  coincides with  the    hyperbolic distance
$\sigma$ on $\mathbb{D}$. Therefore, from  the Yau theorem and the Duren -- Weir result
\begin{equation*}
\beta  (|z|,|w |)\le \beta (z,w),\quad   z, w\in  \mathbb {B}^n,
\end{equation*}
we may deduce the Schwarz lemma for the modulus od a  vector--valued  holomorphic mapping on $M$. In particular, if $M$ is the unit disc $\mathbb{D}$
equipped with the hyperbolic metric, we have the Pavlovi\'{c} result for holomorphic mappings of $\mathbb {D}$ into $\mathbb {B}^n$ in the derivative--free
formulation.
\end{remark}

\end{document}